\documentclass[12pt]{amsart}

\usepackage{breakurl}
\usepackage{graphicx}
\usepackage{amsmath}
\usepackage{amscd}
\usepackage{amsfonts}
\usepackage{amssymb}
\usepackage[dvipsnames]{xcolor}
\usepackage{fullpage}
\usepackage{mathtools}
\usepackage[pagebackref,, colorlinks, citecolor=BrickRed,linkcolor=MidnightBlue, urlcolor=BrickRed]{hyperref}
\usepackage{todonotes}
\usepackage{enumitem}

\usepackage{cleveref}
\usepackage{orcidlink}

\numberwithin{equation}{section}

\newtheorem{theorem}{Theorem}[section]
\newtheorem*{maintheorem}{Main Result}
\newtheorem{lemma}[theorem]{Lemma}
\newtheorem{proposition}[theorem]{Proposition}

\newtheorem{corollary}[theorem]{Corollary}

\theoremstyle{definition}
\newtheorem{example}[theorem]{Example}
\newtheorem{remark}[theorem]{Remark}
\newtheorem{definition}[theorem]{Definition}

\newcommand{\be}{\begin{equation}}
\newcommand{\ee}{\end{equation}}
\newcommand{\bes}{\begin{equation*}}
\newcommand{\ees}{\end{equation*}}

\newcommand{\cH}{\mathcal{H}}

\newcommand{\cK}{\mathcal{K}}
\newcommand{\cL}{\mathcal{L}}

\newcommand{\cU}{\mathcal{U}}

\newcommand{\cW}{\mathcal{W}}

\newcommand{\bC}{\mathbb{C}}

\newcommand{\bN}{\mathbb{N}}
\newcommand{\bQ}{\mathbb{Q}}

\newcommand{\bS}{\mathbb{S}}
\newcommand{\bT}{\mathbb{T}}
\newcommand{\bZ}{\mathbb{Z}}

\newcommand{\ep}{\varepsilon}
\newcommand{\eps}{\varepsilon}

\newcommand{\ol}{\overline}

\newcommand{\id}{\operatorname{id}}

\newcommand{\spn}{\operatorname{span}}

\newcommand{\distance}{\operatorname{d}}

\begin{document}

\title{Dilation distance and the stability of ergodic commutation relations}

 \author{Malte Gerhold}
 \address{M.G., Department of Mathematics, Saarland University, D-66123 Saarbrücken, Germany}
 \email{gerhold@math.uni-sb.de}
\urladdr{\url{https://sites.google.com/view/malte-gerhold/}\newline
  \href{https://orcid.org/0000-0003-4029-1108}{{\orcidlink{}}\,\url{https://orcid.org/0000-0003-4029-1108}}}

 \author{Orr Moshe Shalit}
 \address{O.S., Faculty of Mathematics\\
 Technion - Israel Institute of Technology\\
 Haifa\; 3200003\\
 Israel}
 \email{oshalit@technion.ac.il}
 \urladdr {\href{https://oshalit.net.technion.ac.il/}{\url{https://oshalit.net.technion.ac.il/}}\newline
\href{https://orcid.org/0000-0002-5390-6928}{{\orcidlink{}}\,\url{https://orcid.org/0000-0002-5390-6928}}}


 \thanks{The work of M.~Gerhold was supported by the German Research Foundation (DFG) grant no 465189426;
it was carried out as a postdoctoral researcher at Saarland University, during the tenure of an ERCIM “Alain
Bensoussan” Fellowship Programme at NTNU Trondheim, and as a guest researcher at Saarland University in the scope
of the SFB-TRR 195.
}
\thanks{The work of O.M.~Shalit is partially supported by ISF Grant no.\ 431/20. }

 \subjclass[2020]{47A20 (Primary) 47A13, 46L07 (Secondary)}
 \keywords{Dilations, matrix range, almost $q$-commuting, noncommutative tori}

 \addcontentsline{toc}{section}{Abstract}

\begin{abstract}
We revisit and generalize the notion of {\em dilation distance} $\distance_{\mathrm{D}}(u,v)$ between unitary tuples and study its relation to the natural Haagerup-R{\o}rdam distance $\distance_{\mathrm{HR}}(u,v) = \inf\{\|\pi(u) - \rho(v)\|\}$, where the infimum is taken over all pairs of faithful representations $\pi \colon C^*(u) \to B(\cH)$, $\rho \colon C^*(v) \to B(\cH)$.
We show that $\distance_{\mathrm{HR}}(u,v)\leq 10\distance_{\mathrm{rD}}(u,v)^{1/2}$, where $\distance_{\mathrm{rD}}(u,v)$ is a relaxed dilation distance, improving and extending earlier results. 
For an antisymmetric matrix $\Theta$, we show via a concrete dilation construction that a tuple of unitaries $u$ that almost commutes according to $\Theta$ (i.e., $\|u_\ell u_k - e^{i \theta_{k,\ell}} u_k u_\ell\|$ is small) can be nearly dilated to a tuple of unitaries $v$ that commutes according to $\Theta$ (i.e., $v_\ell v_k - e^{i \theta_{k,\ell}} v_k v_\ell = 0$). 
We show that the dilation can be ``reversed'' by a second application of the dilation construction, which leads to a rotated version of the original tuple. 
Thus, a gauge invariant almost $\Theta$-commuting unitary tuple can be approximated (in some faithful representation) by a $\Theta$-commuting unitary tuple. 
Moreover, when $\Theta$ is ergodic, a $\Theta$-commuting tuple is shown to be {\em almost} gauge invariant, and it follows from the results above that these can be approximated in norm by $\Theta$-commuting tuples. 
In particular, we obtain the following counterpart of Lin's theorem on almost commuting unitaries: if $q \in \bT$ is {\em not} a root of unity, then for every $\eps >0$ there exists $\delta > 0$ such that for every pair of unitaries $u_1,u_2 \in B(\cH)$ for which $\|u_1 u_2 - qu_2 u_1\| < \delta$, there exists two $q$-commuting unitaries $v_1, v_2 \in B(\cH \otimes \ell^2)$ such that $\|v_i - u_i \otimes 1\| < \eps$ ($i=1,2$). 
\end{abstract}

\maketitle

\section{Introduction}\label{sec:introduction}

\subsection{Background}

A fundamental question, which must have occurred to anyone who has ever solved a mathematical problem in practice (i.e., needed to get the numbers right), is: 

\vskip 5pt

\begin{center} {\em Given an object that approximately solves a problem, is it approximately a solution, that is, is this object close to a true solution of the problem?}
\end{center}

\vskip 5pt
A specific version of this question was answered by Hyers \cite{Hyers}, who showed that for every $\eps > 0$ there exists a $\delta > 0$ such that whenever $f \colon X \to Y$ is a map between Banach spaces $X$ and $Y$ such that 
\be\label{eq:stability}
\|f(x_1 + x_2) - f(x_1) - f(x_2) \| < \delta \quad, \quad \textrm{ for all } x_1, x_2 \in X
\ee
there exists a map $A \colon X \to Y$ which is {\em additive} (that is, $A(x_1 + x_2) = A(x_1) + A(x_2)$)  such that 
\[
\|f(x) - A(x) \| < \eps \quad, \quad \textrm{ for all } x \in X; 
\]
(in fact, in this setting $\delta = \eps$ works, but the $\delta,\eps$ formulation is better suited for generalization). 
This is one instance of a classical problem of Ulam \cite{Ulam}, where $X$ and $Y$ are replaced by some group and in \eqref{eq:stability} the operation `$+$' is replaced by group multiplication and the norm is replaced by some metric on the group. 
Roughly, Ulam asked whether or not every almost-homomorphism between metric groups is close to a true homomorphism. 
This problem, which is briefly referred to as {\em stability}, has attracted much attention through the years; 
for example, in the paper \cite{Kazhdan82} Kazhdan showed that every amenable group $G$ is {\em stable}, in the sense that for every $\eps$ there exists $\delta$ such that every $\delta$-almost homomorphism from $G$ into a unitary group is $\eps$-close to a true unitary representation. Moreover, Kazhdan showed  that the (noncommutative) free groups are not stable. 

An operator theoretic variation on this theme was suggested by Halmos \cite{Halmos_Open}, who asked whether for every $\eps$ there exists a $\delta$ such that for every pair of selfadjoint matrices (of arbitrary {\em finite} size) $A,B$, if $A$ and $B$ {\em $\delta$-almost commute}, in the sense that $\|AB - BA\|<\delta$, then $A$ and $B$ are {\em $\eps$-close to commuting selfadjoints}, in the sense that there exist commuting selfadjoint matrices $A'$ and $B'$, such that $\|A - A'\| < \eps$ and $\|B - B'\|< \eps$. 
This was solved affirmatively (for matrices) by Lin \cite{Lin96almost}; for pairs of  selfadjoint operators on an infinite dimensional space the answer is in general negative, as it is negative for three or more selfadjoint matrices \cite{davidson1985almost}. 

There are interesting variations on Halmos's question. 
One can ask the same question but replace the word ``selfadjoint" with ``unitary": {\em is it true that for every $\eps$ there exists $\delta$ such that every pair of $\delta$-almost commuting unitaries is $\eps$-close to a pair of commuting unitaries?}
Voiculescu showed that even if one restricts attention to matrices  then the answer to this question is: no \cite{Voiculescu1983} (see \cite{EL89} for a very elegant proof of this fact). 
Note that this does not contradict Kazhdan's result mentioned above, since having a pair of $\delta$-almost commuting unitaries is a significantly weaker assumption then having a $\delta$-almost homomorphism from the amenable group $\bZ^2$ into a unitary group (in the latter case, the analogue of \eqref{eq:stability} has to hold for every two elements in the group, not just the generators). 
Still, we think of the circle of problems coming out of Halmos's question as part of the broad stability perspective.

\subsection{Objective}

In this paper we attack another variant of the stability, in which the question is whether a tuple of unitaries that almost satisfies a certain relation has a $*$-isomorphic copy that is close to a tuple that truly satisfies the relations. 
In more detail, the problem we solve is as follows. 

Recall that for a real and antisymmetric $d\times d$ matrix $\Theta=(\theta_{k,\ell})$ and a unitary tuple $U = (U_1, \ldots, U_d)$, we say that $U$ is {\em $\Theta$-commuting}, or that $U$ {\em commutes according to $\Theta$}, if 
\be\label{eq:Theta_commuting}
U_\ell U_k=e^{i\theta_{k,\ell}}U_k U_\ell,
\ee
for all $k,\ell = 1, \ldots, d$. 
Alternatively, we put $Q = (\exp(i\theta_{k,\ell}))$, and we say that {\em $U$ is $Q$-commuting}. 
For every $\Theta$ as above there exists a {\em universal} $\Theta$-commuting unitary tuple $u_\Theta$, characterized by the property that for every $\Theta$-commuting unitary tuple $U$ there exists a $*$-homomorphism $\pi \colon C^*(u_\Theta) \to C^*(U)$ such that $\pi(u_{\Theta}) = U$. 

We say that $\Theta$ is {\em ergodic} if the columns of $Q$ generate a dense subgroup of $\bT^d$. 
Thus, for example, if $d = 1$ and $Q = \left(\begin{smallmatrix} 1 & q \\ \bar q & 1 \end{smallmatrix} \right)$, then $\Theta$ is ergodic if and only if $q$ is not a root of unity. 

Given $\delta > 0$ and a matrix $\Theta$ as above, we say that $U \in \cU(d)$ is {\em $\delta$-almost $\Theta$-commuting} if 
\be\label{eq:del_almost_T_comm}
\|U_\ell U_k - e^{i\theta_{k,\ell}}U_k U_\ell\| \leq \delta , 
\ee
for all $k,\ell = 1, \ldots, d$.

\begin{maintheorem}[Theorem \ref{thm:main2} below]\label{thm:mainintro}
Let $\Theta$ be an ergodic real antisymmetric $d \times d$ matrix. 
For every $\ep > 0$, there exists $\delta > 0$, such that for every $\delta$-almost commuting unitary tuple $U$ on $\cH$ there exists a $\Theta$-commuting unitary tuple on $\cH \otimes \ell^2$ such that 
\[
\|U \otimes \mathrm{id}_{\ell^2} - V_\Theta\| < \ep .
\]
\end{maintheorem} 

This theorem is a counterpart of Lin's result \cite[Theorem 2.5]{Lin98} that, if $u$ and $v$ are almost commuting unitaries, then their infinite ampliations $u \otimes \mathrm{id}_{\ell^2}$ and $v \otimes \mathrm{id}_{\ell^2}$ are near commuting unitaries, that is, Lin's result is the above statement for $d = 1$ and $\Theta = 0$. 
For $d>2$, the $\Theta = 0$ analogue of Theorem \ref{thm:mainintro} follows from the main result in the recent paper \cite{lin2024almost}. 
It is interesting to mention that stability results for the case $\Theta \neq 0$ have already been obtained under certain assumptions; see \cite{HuaLin2015,HuaWang2021}. 

The reader might wonder about the infinite ampliation appearing in the theorem. 
The first thing to say here is that it is necessary: the analogous statement is false with the ampliation removed, as can be seen by considering rational rotations tending towards an irrational one. 
The second comment is that this theorem solves the problem we posed, of approximating an almost $\Theta$-commuting tuple by a $\Theta$-commuting tuple {\em up to isomorphism}. 
It is interesting that in some cases our approximation result can be bootstrapped to provide approximation without ampliation; see Example \ref{ex:standard_q}. 
Finally, we note that approximation up to isomorphism (or up to infinite ampliation) is of interest and may lead to striking consequences, see for example \cite{HR95}.

Theorem \ref{thm:mainintro} falls within a broader long term project of applying dilation techniques to stability questions, building on and refining our work from \cite{GPSS21,GS20,GS23}. 
We now turn to describing our framework.

\subsection{Distances on the space of unitary tuples}\label{subsec:distances} 

At this point we begin the technical discussion, and introduce the various measurements of ``closeness" to be used in our investigation of stability. 
The metrics that we define below will be shown in the next section to be equivalent in an explicitly quantitative sense (see Theorem \ref{thm:dHRleqdD}); this will be key in obtaining the main result. 

For tuples $u = (u_1, \ldots, u_d),v = (v_1, \ldots, v_d)\in B(\cH)^d$, we define
\[
\|u - v\| = \max\{\|u_i - v_i\| \colon i=1, \ldots, d\}. 
\]
If $u$ and $v$ are two $d$-tuples of unitaries, then we say that $u$ and $v$ are {\em equivalent}, and we write $u \sim v$, if there is a $*$-isomorphism $\phi \colon C^*(u) \to C^*(v)$ such that $\pi(u) = v$, by which we mean that $\pi(u_i) = v_i$ for all $i=1, \ldots, d$. 
Let $\cU(d)$ denote the space of all equivalence classes of $d$-tuples of separably acting unitaries. 
We shall sometimes identify a tuple with the equivalence class it belongs to. 

In \cite{GPSS21}, we considered three distance functions on the space $\cU(d)$. 
Perhaps the most natural way to measure a distance between $u, v \in \cU(d)$ is given by the following distance function which we call
{\em the Haagerup-R{\o}rdam-distance}:
\[
\distance_{\mathrm{HR}}(u,v):=\inf\left\{\left\|u'-v'\right\|\colon u', v' \in B(\cH)^d,  u \sim u' \textrm{ and } v \sim v'\right\}. 
\]
It is worth pointing out that that $\distance_{\mathrm{HR}}(u,v) = 0$ if and only if $u \sim v$ (see \cite[Proposition 2.3]{GPSS21}; alternatively this can be shown directly be constructing a representation into an appropriate direct sum).

Another distance function is given by the {\em matrix range distance}, defined by 
\[
\distance_{\mathrm{mr}}(u,v):=\distance_{\mathrm H}\bigl(\cW(u),\cW(v)\bigr) .
\]
Here, the matrix range of a tuple $A \in B(\cH)^d$ is given by $\cW(A) = \sqcup_{n=1}^\infty \cW_n(A)$ where
\[
\cW_n(A) = \{\phi(A) \colon \operatorname{UCP}(C^*(A), M_n)\},
\]
and $\distance_{\mathrm H}\bigl(\cW(u),\cW(v)\bigr)$ is the Hausdorff distance between the matrix ranges, by which we mean the supremum over the Hausdorff distances $\sup_n \distance_{\mathrm H}\bigl(\cW_n(u),\cW_n(v)\bigr)$. 

It will be convenient to define also the \emph{one-sided matrix range distance} between tuples $A \in B(\cH)^d$ and $B \in B(\cK)^d$:
\[
  \distance_{\mathrm{mr}}(A\to B):= \sup_n \sup_{X\in \cW(A)}\mathop{\inf\vphantom{sup}}_{Y\in \cW_n(B)} \|X-Y\| .
\]
Recall that the one sided matrix range distance has been discussed before as $\delta_{\cW(A)}(\cW(B))=\distance_{\mathrm{mr}}(B\to A)$ by Davidson, Dor-On, Solel and the second named author \cite{DDSS17}. The {matrix range distance} is then given by 
\[
\distance_{\mathrm{mr}}(A,B)= \max (\distance_{\mathrm{mr}}(A\to B),\distance_{\mathrm{mr}}(A\to B)) .
\]
Note that the matrix range distance is a metric only when considered on classes of certain {\em rigid} tuples, for example on $\mathcal U(d)$. 
For general operator tuples it is not true that the matrix range determines the tuple up to $*$-isomorphism.

A third measure of difference between tuples was given by the {\em dilation distance}, which was defined in \cite{GPSS21} as follows. 
First, we define the dilation constant $c(u,v)$. 
Given two unitary tuples $u,v$ and a positive real number $c$, we write $u \prec cv$ if there exist two Hilbert spaces $\cH \subseteq \cK$ and two operator tuples $U \in B(\cH)^d$ and $V \in B(\cK)^d$, such that $u \sim U$, $v \sim V$ and 
\[
U = P_\cH cV \big|_\cH. 
\]
By Stinespring's theorem, $u \prec cv$ if and only if there is a unital completely positive (UCP) map from the operator system generated by $v$ to the operator system generated by $u$, that maps $cv$ to $u$. 
For $u,v \in \cU(d)$, we put 
\[
c(u,v) = \inf\{c \colon  u \prec cv\}. 
\]
The {\em dilation distance} is then defined by 
\[
\distance_{\mathrm{D}}(u,v):= \log \max\bigl\{c(u,v), c(v,u)\bigr\} .
\]
Various aspects of this notion were studied in \cite{GPSS21}, where among other things it was compared to $\distance_{\mathrm{mr}}$ and used to prove that the noncommutative tori form a continuous field of C*-algebras in a strong sense, recovering results of Haagerup and R{\o}rdam \cite{HR95} and of Gao \cite{Gao18}. 
In \cite{GS23} a variant of the dilation distance suited for groups of unitary operators was employed to show that the infinite multiplicity versions of the momentum and position operators from quantum mechanics can be boundedly perturbed to a strongly commuting pair of operators (again, recovering and somewhat improving results from \cite{Gao18,HR95}). 
In this paper, we modify the definition of dilation distance in order to give some more flexibility.

\begin{definition}
  For $A\in B(\cH)^d$ and $B \in B(\cK)^d$ $d$-tuples of operators, we define
  \begin{itemize}[beginpenalty=10000]
  \item the \emph{one-sided relaxed dilation distance}
    \(\displaystyle\distance_{\mathrm{rD}}(A\to B):= \inf_{\substack{\pi(A)\sim A\\ \Psi(B)\prec B}} \|\pi(A)-\Psi(B)\|\) ,
  \item the \emph{relaxed dilation distance}
    \(\displaystyle\distance_{\mathrm{rD}}(A,B):= \max (\distance_{\mathrm{rD}}(A\to B),\distance_{\mathrm{rD}}(A\to B))\) .
  \end{itemize}
\end{definition}
The notation $\distance_{\mathrm{rD}}(A\to B) < \delta$ should be understood as ``$A$ almost dilates to $B$ with error $\delta$". 
Note that the $*$-isomorphic map $\pi$ is not really necessary in the definition of one-sided relaxed dilation distance; indeed any unital completely isometric (UCI) map from the operator system $S_A$ generated by $A$ to $B(\cL)$ has a UCI inverse $\pi^{-1}\colon \pi(S_A)\to S_A$ which extends to a UCP map $\Phi\colon B(\cL) \to B(\cH)$. Since $\|\pi(A)-\Psi(B)\|\geq \|\Phi(\pi(A)-\Psi(B))\|=\|A-\Phi\Psi(B)\|$, we see that
\[\inf_{\substack{\pi(A)\sim A\\ \Psi(B)\prec B}} \|\pi(A)-\Psi(B)\|=\inf_{\Psi(B)\prec B} \|A-\Psi(B)\|.\]

\begin{theorem}[This is {\cite[Theorem 5.7]{DDSS17}} with minimal adjustments]
Let $A \in B(\cH)^d$ and $B \in B(\cK)^d$
such that
$\distance_{\mathrm{mr}}(B\to A)=\delta_{\cW(A)}
(\cW(B)) = r$.
Then there is a UCP map $\Psi$ of $S_A$ into $B(\cK)$ such that
$\|\Psi(A) - B\| \leq r$.
\end{theorem}

We rephrase this as $\distance_{\mathrm{rD}}(A\to B)\leq \distance_{\mathrm mr}(A\to B)$. The reverse inequality also holds, so we get a characterization of matrix range distance in terms of dilations.

\begin{lemma} 
Let Let $A \in B(\cH)^d$ and $B \in B(\cK)^d$. Then
\[
\distance_{\mathrm mr}(A\to B)= \distance_{\mathrm{rD}}(A\to B).
\]
\end{lemma}
\begin{proof}
  We only need to show the inequality $\distance_{\mathrm mr}(A\to B) \leq \distance_{\mathrm{rD}}(A\to B)$. For any UCP map $\Psi\colon S_B\to B(\cH)$ we find that $X=\Phi(A) \in \cW(A)$ with $\Phi\colon S_A\to M_n$ implies
  \[\distance(X,\cW(B))\leq \|X-\Phi(\Psi(B))\|= \|\Phi(A-\Psi(B))\|\leq \|A-\Psi(B)\|.\]
  By taking the infimum over all such $\Psi$, we obtain the desired inequality. 
\end{proof}
From here on we shall use freely the above relationship between the dilation distance and the matrix range distance.

\begin{corollary}\label{cor:deqd}
The matrix range distance $\distance_{\mathrm{mr}}$ and the relaxed dilation distance $\distance_{\mathrm{rD}}$ are equal.
\end{corollary}

\begin{remark}
  The reader might wonder why we have two names for the same distance. On the one hand, it is for the practical reason that the proofs in this section leading to the equality can be better structured if we distinguish them at first. But also, $\distance_{\mathrm{rD}}$ and $\distance_{\mathrm{mr}}$ each come from a slightly different point of view, where $\distance_{\mathrm{rD}}$ is directly related to dilations and therefore suggests the method of attack that we use to prove our main result in this article, namely to construct dilations, not to compute matrix ranges. Finally, in hindsight, we realized that the original definition of the dilation distance $\distance_{\mathrm{D}}$ proposed in \cite{GPSS21} was not optimal. Although its domination of $\distance_{\mathrm{HR}}$ is a powerful resource, $\distance_{\mathrm{D}}$ is best suited for gauge invariant tuples; on general operator tuples it behaves badly. For example, the dilation distance between two distinct complex numbers $z_1,z_2\in\mathbb T$ is always infinite. So we want to promote the idea to use $\distance_{\mathrm{rD}}$ instead, whenever distance shall be measured via dilations.
\end{remark}
\begin{remark}
In the next section we will see that if $u^{(k)}$ is a sequence of tuples then $\lim_{k \to \infty} \distance_{\mathrm{mr}}(u^{(k)},v) = 0$ implies $\lim_{k \to \infty} \distance_{\mathrm{HR}}(u^{(k)},v) = 0$ (see Theorem \ref{thm:dHRleqdD}), thus
\[
\sup_n \distance_{\mathrm H}\bigl(\cW_n(u^{(k)}),\cW_n(v)\bigr) \xrightarrow{k \to \infty} 0 \,\, \Longleftrightarrow \,\, \distance_{\mathrm{HR}}(u^{(k)},v) \xrightarrow{k \to \infty} 0 . 
\]
It is interesting to note that {\em levelwise} convergence of the matrix ranges, that is, 
\[
\distance_{\mathrm H}\bigl(\cW_n(u^{(k)}),\cW_n(v)\bigr) \xrightarrow{k \to \infty} 0 \quad, \quad \textrm{ for all } n,
\] 
is equivalent to convergence in the sense of continuous fields of C*-algebras, that $\|p(u^{(k)})\| \to \|p(v)\|$ for every $*$-polynomial $p$; see Theorems 2.3 and 2.4 in \cite{GS21}. 
This was used in \cite{GS21} to identify the numerical range and the matrix range of tuples of random matrices. 
\end{remark}

\section{The dilation distance dominates the Haagerup-R{\o}rdam distance}

The following theorem is proved very closely along the lines of \cite[Theorem 2.6]{GPSS21}. 
However, we need to adapt the proof for our relaxed notion of dilation distance, and so we take this opportunity to introduce several small improvements in order to lower the constant.

\begin{theorem}\label{thm:dHRleqdD}
For $u,v\in\mathcal U(d)$,
  \(\distance_{\mathrm{HR}}(u,v)\leq 10\distance_{\mathrm{rD}}(u,v)^{1/2}.\)
\end{theorem} 

\begin{proof}
We put $\delta:=\distance_{\mathrm{rD}}(u,v)$ and assume that
\be
u \sim
\begin{pmatrix}
v' & x \\ y & z 
\end{pmatrix} \quad \textrm{and} \quad
v \sim 
\begin{pmatrix}
u' & r \\ s & t 
\end{pmatrix}
\ee
with $\|v-v'\|,\|u-u'\|\leq \delta$.

As $u$ and $v$ are unitary, we find that
\[y^*y=1-v'^*v'= v^*(v-v') + (v-v')^*v'\]
and, since $\|v'\|\leq 1$,
$\|y\|\leq \sqrt{2\delta}$.
Analogous reasoning yields $\|x\| , \|s\| , \|r\| \leq \sqrt{2\delta}$. With the block matrices
\[E:=
\begin{pmatrix}
  v'-v &x\\
  y & 0
\end{pmatrix},\quad
F:=
\begin{pmatrix}
  u'-u &r\\
  s&0
\end{pmatrix},\]
we obtain
\begin{align}
\label{eq:equiv}
  u\sim v\oplus z + E,\quad v\sim u\oplus t +F.
\end{align}
The norm of $E$ can be estimated as
\[\|E\|\leq \max (\|x\|,\|y\|) +  \|(v'-v)\| \leq  \sqrt{2\delta} + \delta,
\]
and, analogously, $\|F\|\leq \sqrt{2\delta}+\delta$.

We will need to be careful and track the identifications made. 
But first, we replace every operator $a$ appearing above with the infinite ampliation $a \oplus a \oplus \cdots$, so we may assume that $u,v,x,y$, etc., are all given as concrete operators acting on an infinite dimensional Hilbert space $\cH$. 
So the equivalences in \eqref{eq:equiv} are due to $*$-isomorphisms $\pi \colon C^*(u) \subset B(\cH) \to B(\cH \oplus \cH)$ and $\rho \colon C^*(v) \subset B(\cH) \to B(\cH \oplus \cH)$ such that 
\[
\pi(u) =v\oplus z + E \quad \textrm{and} \quad
\rho(v) = u\oplus t + F.
\]
In fact, by applying the standard combination of Arveson's extension theorem and Stinespring's dilation theorem, we may assume that $\pi$ and $\rho$ are $*$-homomorphisms defined on all of $B(\cH)$. 
Letting $\pi^{(k)}$ and $\rho^{(k)}$ denote the ampliations of the representations, we obtain
\begin{align*}
\pi^{(2)}\rho(v) = \pi(u)\oplus \pi(t) + \pi^{(2)}(F) = v\oplus z\oplus \pi(t) + E\oplus 0_{\cH^2} + \pi^{(2)}(F)\in B(\cH^4) . 
\end{align*}
On the other hand, we find that 
\begin{align*}
  \pi^{(4)} \rho^{(2)}\pi(u)
  &= \pi^{(4)} \rho^{(2)} ( v \oplus z ) + \pi^{(4)} \rho^{(2)} (E) \\
  &= \pi^{(4)} (u \oplus t \oplus \rho(z)) + \pi^{(4)}(F\oplus 0_{\cH^2}) +  \pi^{(4)} \rho^{(2)} (E)\\
  &= v \oplus z \oplus \pi(t) \oplus \pi^{(2)} \rho(z) + E\oplus 0_{\cH^6} + \pi^{(2)}(F)\oplus 0_{\cH^4} +  \pi^{(4)} \rho^{(2)} (E)\in B(\cH^8)   . 
\end{align*}
To summarize a little more concisely what we found: 
\[
v \sim (v\oplus z\oplus \pi(t)) + R \in B(\cH \oplus \cH \oplus \cH^2), 
\]
and
\[
u \sim (v\oplus z\oplus \pi(t) \oplus \pi^{(2)}\rho(z)) + S\in B(\cH \oplus \cH \oplus \cH^2 \oplus \cH^4) , 
\]
where $\|R\| + \|S\|\leq 3\|E\|+2\|F\|\leq 5(\sqrt{2\delta}+\delta)$. 
Applying a permutation on the direct summands, it is convenient to rewrite this as follows: 
\[
v \sim V := (v\oplus  \pi(t) \oplus  z) + R' \in B(\cH  \oplus \cH^2 \oplus \cH ), 
\]
and
\[
u \sim U: =   (v \oplus \pi(t) \oplus z \oplus \pi^{(2)}\rho(z))  + S' \in B(\cH \oplus \cH^2  \oplus \cH \oplus \cH^4) .
\]
If we restrict $\pi^{(2)} \rho$ to the C*-algebra generated by $z$ (which we have assumed to have infinite multiplicity), then we are in the situation of Voiculescu's theorem, which tells us that the representations $\id$ and ${\id} \oplus \pi^{(2)} \rho$ are approximately unitarily equivalent, that is, there is a sequence of unitaries $w_n \colon \cH \to \cH \oplus \cH^4$ such that $\lim_{n\to \infty} \|w_n z w_n^* - z \oplus \pi^{(2)} \rho(z) \| = 0$ (see \cite[Corollary II.5.5]{DavBook}). 
Then letting $W_n = I_{\cH} \oplus I_{\cH^2} \oplus w_n$, we have 
\[
  \limsup_{n\to \infty} \left\| W_n V W_n^*- U \right\| \leq \|R'\| + \|S'\| = \|R\| + \|S\|=5(\sqrt{2\delta}+\delta).
\]
Since $U \sim u$ and $W_n V W_n^* \sim v$ for all $n$, we conclude that 
\[
\distance_{\mathrm{HR}}(u,v) \leq 5(\sqrt{2\delta}+\delta). 
\]
If $\sqrt{\delta}\geq \frac{1}{5}$, the estimate claimed in the theorem is trivial because $\distance_{\mathrm{HR}}$ is globally bounded by $2$. For $\sqrt{\delta}<\frac{1}{5}$, we have $5\delta < \sqrt{\delta}$ and we can also conclude the claimed inequality
\[\distance_{\mathrm{HR}}(u,v) \leq 5(\sqrt{2\delta}+\delta)\leq (5\sqrt 2 + 1)\sqrt{\delta}\leq 10\sqrt\delta.\] 
\end{proof}

\begin{remark}\label{rem:inf_amp}
If $\distance_{\mathrm{HR}}(u,v) < r$ then there are faithful representations $\pi \colon C^*(u) \to B(\cH)$ and $\rho:C^*(v) \to B(\cH)$ such that $\|\pi(u)  - \rho(v)\| < r$. 
We may as well assume that $\pi$ and $\rho$ have infinite multiplicity. 
By Voiculescu's theorem then both $\pi$ and $\rho$ are approximately unitary equivalent to the infinite ampliation. 
We therefore conclude that whenever $u$ is a unitary tuple on $\cH$ and $\distance_{\mathrm{HR}}(u,v) < r$, then there exists $\rho:C^*(v) \to B(\cH \otimes \ell^2)$ such that $\|u \otimes \mathrm{id}_{\ell^2} - \rho(v) \|<r$. 
We shall make use if this observation repeatedly below. 
\end{remark}

\begin{corollary}
The metrics $\distance_{\mathrm{HR}}$ and $\distance_{\mathrm{mr}}=\distance_{\mathrm{rD}}$ are equivalent on $\mathcal U(d)$. 
\end{corollary}

\begin{proof}
By Corollary \ref{cor:deqd},  $\distance_{\mathrm{mr}}=\distance_{\mathrm{rD}}$. 
The straightforward observation that $\distance_{\mathrm{mr}}\leq \distance_{\mathrm{HR}}$ is made in \cite[Lemma 2.2]{GPSS21}.
\end{proof}

\section{Dilating almost $\Theta$-commuting tuples to $\Theta$-commuting tuples}\label{sec:dilation}

\subsection{The basic dilation construction}

\begin{lemma}\label{lem:DACU}
Let $q \in \bT$, let $\delta \in (0,1)$ and let $u$ and $v$ be two unitaries on $\cH$ such that $\|vu - quv\| \leq \delta$.
Then there exist two unitaries $\tilde{u},\tilde{v}$ such that $\tilde{v}\tilde{u} = q\tilde{u}\tilde{v}$ and 
\[
\distance_{\mathrm{rD}}((u,v) \to (\tilde{u},\tilde{v})) < \sqrt{\delta}. 
\]
\end{lemma}
\begin{proof}
Let $S$ be the bilateral shift on $\ell^2(\mathbb Z)$ given by $S e_k=e_{k+1}$, where $\{e_k\}_{k \in \bZ}$ is the standard basis of $\ell^2(\bZ)$.
Let $p_k$ be the projection onto $\mathbb C e_k$ and define 
\[
\widetilde u:=u\otimes S
\] 
and 
\[
\widetilde v:= \sum q^k u^{k}vu^{-k}\otimes p_k .
\]
Then $\widetilde u,\widetilde v$ are $q$-commuting unitaries:
\begin{align*}
q \widetilde u\widetilde v &= \sum q^{k+1} u^{k+1}vu^{-k} \otimes Sp_k \\
&= \sum q^{k+1} u^{k+1}vu^{-k}\otimes p_{k+1} S \\
&= \sum q^{k}u^{k}vu^{-k+1}\otimes p_k S = \widetilde v \widetilde u .
\end{align*}
Put $\xi_N:=\frac{1}{\sqrt{2N+1}}\sum_{k=-N}^N e_k$, which is a unit vector in $\ell^2(\mathbb Z)$. 
We compress $\widetilde u, \widetilde v$ with respect to the isometry $\iota:=\mathrm{id}\otimes\xi_N\colon \cH \to \cH\otimes \ell^2(\mathbb Z)$, i.e.,
\[
\iota h = \frac{1}{\sqrt{2N+1}}h\otimes \sum_{k=-N}^N e_k \quad\text{and}\quad \iota^*\sum_{k=-\infty}^\infty h_k\otimes e_k=\frac{1}{\sqrt{2N+1}} \sum_{k=-N}^N h_k
\]
and obtain
\[
\iota^*\widetilde u\iota =u\langle S \xi_N,\xi_N\rangle = \frac{2N}{2N+1} u,\quad
\iota^*\widetilde v\iota =\frac{1}{2N+1}\sum_{k=-N}^N  q^k u^{k}vu^{-k}.
\]
Now suppose that $N$ is chosen such that $\frac{N+1}{2} < \delta^{-\frac{1}{2}} < 2N+1$. We obtain
\[
\|u-\iota^*\widetilde u\iota\| = \frac{1}{2N+1} < \delta^{\frac{1}{2}}.
\]
To see that also $v$ and $\iota^*\widetilde v \iota$ are close, note that, for $k>0$,  
\begin{align*}
\|v-q^k u^{k}vu^{-k}\| &|\leq \sum_{m=0}^{k-1}\|q^m u^{ m}vu^{-m} - q^{m+1}u^{m+1}vu^{-m-1}\| \\ 
&= \sum_{m=0}^{k-1}\|vu - quv\| \leq  k\delta ;
\end{align*}
the calculation for $k<0$ is analogous.
Therefore, we can conclude
\begin{align*}
  \|v-\iota^*\widetilde v\iota\|
  &= \frac{1}{2N+1}\left\|\sum_{k=-N}^N v-q^k u^{k}vu^{-k} \right\|\\
  &\leq \frac{1}{2N+1}\sum_{k=-N}^N \|v-q^k u^{k}vu^{-k}\|\\
  &\leq \frac{1}{2N+1}\sum_{k=-N}^N |k|\delta\\
  &= \frac{N(N+1)}{2N+1}\delta < \frac{N+1}{2}\delta< \delta^{\frac{1}{2}}.
\end{align*}
Since $a \mapsto \iota^* a\iota$ is a UCP map, this completes the proof. 
\end{proof}

\begin{remark}
We point out that by tensoring $\tilde{u}, \tilde{v}$ with a universal commuting unitary pair and using arguments from Proposition 2.4 in \cite{GPSS21}, one can find two unitaries $\hat{u},\hat{v}$ such that $\hat{v}\hat{u} = q\hat{u}\hat{v}$ and 
\[
(u,v) \prec c(\hat{u}, \hat{v}) 
\]
with the constant $c=1 + 2\sqrt{2\delta}$. 
Since this will not be needed we omit the details. 
\end{remark}

\begin{lemma}\label{lem:Theta-commuting_dilation}
Let $Q = (\exp(i\theta_{k,\ell}))$ where $\Theta=(\theta_{k,\ell})$ is a real $d\times d$ antisymmetric matrix, and let $\delta \in (0,1)$. 
Let  $U\in B(\cH)^d$ be a unitary $d$-tuple such that for some $m \in \{2, \ldots, d\}$
\begin{itemize}
\item $U_1,\ldots, U_{m-1}$ commute according to the submatrix $(\theta_{k,\ell})_{k,\ell=1}^{m-1}$
\item $\|U_iU_{j}-q_{j,i}U_i U_j\| \leq \delta$ for $1\leq i,j\leq {d}$.
\end{itemize}
Then there exists a Hilbert space $\cK$, an isometry $\iota\colon \cH \to \cK$ and a unitary $d$-tuple $\hat U\in B(\cK)^d$ such that
\begin{itemize}
\item $\hat U_1,\ldots, \hat U_{m}$ commute according to the submatrix $(\theta_{k,\ell})_{k,\ell=1}^{m}$
\item $\|\hat U_i\hat U_{j}-q_{j,i}\hat U_i \hat U_j\| \leq   \delta$ for $1\leq i,j\leq {d}$
\item $\|U-\iota^* \hat U \iota\| < \sqrt{\delta}$.
\end{itemize}   
\end{lemma}

\begin{proof}
The proof is modeled on the proof of Lemma \ref{lem:DACU}, where $U_m$ plays the role of $u$ and $U_j$ ($j \neq m$) all play the role of $v$. 
Let $S$ be the bilateral shift on $\ell^2(\mathbb Z)$ given by $S e_k=e_{k+1}$, where $\{e_k\}_{k \in \bZ}$ is the standard basis of $\ell^2(\bZ)$.
Let $p_k$ be the projection onto $\mathbb C e_k$ and define 
\[
\hat U_j:=
\begin{cases}
  U_m \otimes S & j=m,\\
  \sum q_{j,m}^kU_m^kU_jU_m^{-k}\otimes p_k & j\neq m.
\end{cases}
\]
For $i,j\neq m$, it is easy to see that $\|\hat U_i\hat U_{j}-q_{j,i}\hat U_j \hat U_i\|= \| U_i U_{j}-q_{j,i} U_j  U_i\|$.
The exact same calculations as in Lemma \ref{lem:DACU} with $u=U_m, v=U_j,q:=q_{j,m}$ for $j\neq m$ show that $\hat U_m$ has the desired commutation relation with all $\hat U_j$ and also yield the isometry $\iota$ such that $\|U-\iota \hat U \iota^*\|<\sqrt\delta$.
\end{proof}

\begin{corollary}\label{cor:Theta-commuting_dilation}
Let $Q = (\exp(i\theta_{k,\ell}))$ where $\Theta=(\theta_{k,\ell})$ is a real $d\times d$ antisymmetric matrix. Assume that $U \in \cU(d)$ is $\delta$-almost $\Theta$-commuting. 
Then there  exists a $\Theta$-commuting $V_\Theta \in \cU(d)$ such that $\distance_{\mathrm{rD}}(U\to V_\Theta) < (d-1)\sqrt{\delta}$.   
\end{corollary}

\begin{proof}
We iterate \Cref{lem:Theta-commuting_dilation} $(d-1)$ times for $m=2,\ldots d$, and the unitary tuple $\hat U$ obtained in the last step is the desired $\Theta$-commuting tuple $V_\Theta$. 
\end{proof}

\subsection{Gauge invariant tuples and applications to universal tuples}

A unitary tuple $U$ is called \emph{gauge invariant} if $U\sim \lambda U:=(\lambda_1U_1,\ldots \lambda_d U_d)$ for every $\lambda \in \mathbb T^d$. Clearly, $U$ is gauge invariant if and only if $U\sim U\otimes z:=(U_1\otimes z_1,\ldots, U_d\otimes z_d)$, where $z$ is the universal commuting unitary tuple. 

If $U$ is a $\Theta$-commuting unitary tuple, then $U \sim u_\Theta$ is the universal $\Theta$-commuting tuple if and only if it is gauge invariant (see \cite[Lemma 4.1]{GPSS21}).

\begin{lemma}\label{lem:gi}
If $U$ is gauge invariant then the $\Theta$-commuting dilation $V_\Theta$ constructed in Corollary \ref{cor:Theta-commuting_dilation} is also gauge invariant. 
\end{lemma}
\begin{proof}
This follows from the construction. 
\end{proof}

Given $\Theta=(\theta_{k,\ell})$ and $Q = (\exp(i\theta_{k,\ell}))$ as above and $\delta > 0$, we let $u_{\Theta,\delta}$ be the universal unitary tuple $(u_1, \ldots, u_d)$ that satisfies
\be\label{eq:almost_q_comm}
\|u_\ell u_k - e^{i\theta_{k,\ell}}u_k u_\ell \| \leq \delta,
\ee
for all $ k,\ell = 1, \ldots, d$.
Clearly, $u_{\Theta, \delta}$ is gauge invariant and $\delta$-almost $\Theta$-commuting. 
Henceforth we shall ignore the care $\delta \geq 1$ since in that case the assumption is not very exciting and the result trivial. 

\begin{theorem}\label{thm:dist_uthetadelta}
The universal $\Theta$-commuting unitary tuple $u_\Theta$ and the universal $\delta$-almost $\Theta$-commuting unitary tuple $u_{\Theta,\delta}$ satisfy
\[
{\mathrm d}_{\mathrm{HR}}(u_\Theta, u_{\Theta,\delta}) \leq 10 \sqrt{(d-1)\sqrt{\delta}} .
\]
\end{theorem}

\begin{proof}
It is clear that $u_\Theta \prec u_{\Theta,\delta}$ (because one way to define $u_{\Theta,\delta}$ is to take the direct sum over all tuples that $\Theta$-commute up to an error of at most $\delta$, and clearly $u_\Theta$ satisfies this), thus $\distance_{\mathrm{rD}}(U_\Theta \to U_{\Theta,\delta}) \leq 1$. 
On the other hand, by Corollary \ref{cor:Theta-commuting_dilation} and Lemma \ref{lem:gi}, we have that $\distance_{\mathrm{rD}}(U_{\Theta, \delta} \to U_\Theta)< (d-1)\sqrt{\delta}$. 
Combining the one sided estimates, we have 
$\distance_{\mathrm{rD}}(u_\Theta,u_{\Theta,\delta}) < (d-1)\sqrt{\delta}$. 
Invoking Theorem \ref{thm:dHRleqdD} we obtain the desired result. 
\end{proof}

In \cite[Section 5]{GPSS21} it was shown that the family $\{u_\Theta\}$ is continuous in $\cU(d)$ with respect to any of the metrics, in the sense that if $\Theta$ and $\Theta'$ are close then ${\mathrm d}_{\mathrm{HR}}(u_\Theta,u_{\Theta'})$ is small. 
Since a $\Theta'$-commuting tuple is almost $\Theta$-commuting, we can recover this result with a somewhat weaker quantitative control.

\begin{theorem}\label{thm:contQ}
For the universal $\Theta$ and $\Theta'$ commuting unitary tuples, we have 
\[
{\mathrm d}_{\mathrm{HR}}(u_\Theta, u_{\Theta'}) \leq 10 \sqrt{(d-1)\sqrt{\delta}} .
\]
for $\delta = \max\{|q_{k,\ell} - q'_{k,\ell}| \colon k,\ell = 1, \ldots\}$. 
\end{theorem}
\begin{proof}
Note that if $u,v$ are $q$-commuting unitaries and $q' \in \bT$, then 
\begin{align*}
\|vu - q' uv\| &= \|vu - quv\| + \|quv - q'uv\| \\
&= 0 + |q-q'|\|uv\| \\
&= |q-q'|. 
\end{align*}
It follows that $u_{\Theta'} = (u_1, \ldots, u_d)$ satisfies \eqref{eq:almost_q_comm}. 
by Corollary \ref{cor:Theta-commuting_dilation} and Lemma \ref{lem:gi}, we have that $\distance_{\mathrm{rD}}(U_{\Theta'} \to U_\Theta)< (d-1)\sqrt{\delta}$. 
Symmetrically, we have that $\distance_{\mathrm{rD}}(U_{\Theta} \to U_{\Theta'})< (d-1)\sqrt{\delta}$. 
Once more, Theorem \ref{thm:dHRleqdD} gives the desired result. 
\end{proof}

\section{Approximating almost $\Theta$-commuting tuples by $\Theta$-commuting tuples}\label{sec:approximation}

In this section our goal is to prove that, under certain assumptions, an almost $\Theta$-commuting unitary tuple $U \in B(\cH)^d$ can be approximated with a $\Theta$-commuting tuple in the sense that there exists a $\Theta$-commuting unitary tuple $V_\Theta \in B(\cH \otimes \ell^2)$ such that $\|U \otimes \mathrm{id}_{\ell^2} - V_\Theta\|$ is small\footnote{Note the difference from Theorem \ref{thm:dist_uthetadelta}: here we are interested in approximating a particular given almost $\Theta$-commuting tuple, and not the universal one.}. 
By Theorem \ref{thm:dHRleqdD} and Remark \ref{rem:inf_amp}, it suffices to find some $\Theta$-commuting $V_\Theta \in \cU(d)$ such that 
$\distance_{\mathrm{rD}}(V_\Theta,U)$ is small. 
By Corollary \ref{cor:Theta-commuting_dilation}, we know that given $U$ that is almost $\Theta$-commuting, we can ``almost dilate" $U$ to a tuple $V_\Theta$ such that $\distance_{\mathrm{rD}}(U\to V_\Theta)$ is small. 
Thus, our strategy now will be to start from the dilation $V_\Theta$ given the proof of \ref{cor:Theta-commuting_dilation}, and to dilate it further to an operator that we hope to relate to $U$ somehow. 

\subsection{The basic reverse dilation construction}

\begin{lemma}\label{lem:DACU_inverse}
Let $q$, $u$ and $v$ as in Lemma \ref{lem:DACU}, and let $\tilde{u},\tilde{v}$ be the unitaries constructed in the proof of that lemma.
Define the unitaries $\tilde{\tilde{u}}$ and $\tilde{\tilde{v}}$ on $\cH \otimes \ell^2(\bZ) \otimes \ell^2(\bZ)$ by
\[
\tilde{\tilde{u}} = \tilde{u} \otimes S = u \otimes S \otimes S
\]
and 
\begin{align*}
\tilde{\tilde{v}} 
&= \sum_m q^{-m} (u \otimes \mathrm{id})^{-m} \tilde{v} (u \otimes \mathrm{id})^{m} \otimes p_m \\
&= \sum_{m,k}  q^{k-m} u^{k-m}vu^{m-k}\otimes p_{k} \otimes p_m  .
\end{align*}
Then the pair $(\tilde{\tilde{u}}, \tilde{\tilde{v}})$ is equivalent to a direct sum $\oplus_{(\lambda, \mu)\in \Lambda} (\lambda u,\mu v)$ where $\Lambda$ is a sequence of points in $\bT^2$. 
Moreover
\[
\distance_{\mathrm{rD}}((\tilde{u},\tilde{v}) \to (\tilde{\tilde{u}}, \tilde{\tilde{v}})) < \sqrt{\delta}. 
\]
\end{lemma}
\begin{proof}
For every $\ell \in \bN$ let 
\[
\cH_\ell = \ol{\spn}\{\cH \otimes e_k \otimes e_{k + \ell} \colon k \in \bZ\} \subseteq \cH \otimes \ell^2(\bZ) \otimes \ell^2(\bZ).
\]
Then clearly every $\cH_\ell$ is a reducing subspace for $(\tilde{\tilde{u}}, \tilde{\tilde{v}})$ and 
\[
\cH \otimes \ell^2(\bZ) \otimes \ell^2(\bZ) = \oplus_{\ell \in \bZ} \cH_\ell . 
\]
We can think of $\cH_\ell$ as $\cH \otimes \ell^2(\bZ)$, and as such it carries a natural shift 
\[
\mathrm{id}_\cH \otimes S_\ell \colon h \otimes e_k \otimes e_{k+\ell} \mapsto h \otimes e_{k +1}\otimes e_{k+1+\ell} .
\]
On $\cH_\ell$ the pair $\tilde{\tilde{u}}, \tilde{\tilde{v}}$ reduces to the pair comprised of $u \otimes S_\ell$ and $q^{-\ell} u^{-\ell} v u^\ell \otimes \mathrm{id}_{\ell^2(\bZ)}$. 
We see that the pair $\tilde{\tilde{u}}, \tilde{\tilde{v}}$ is unitary equivalent to the direct sum
\[
\bigoplus_\ell (u \otimes S_\ell, v \otimes q^{-\ell}\mathrm{id}) 
\]
which readily implies that it is equivalent to $\oplus_{(\lambda, \mu) \in \Lambda} (\lambda u,\mu v)$ where $\Lambda$ is a sequence of point in $\bT \times \{q^\ell \colon \ell \in \bZ\} \subseteq \bT^2$.

Now we show that $\tilde{\tilde{u}}, \tilde{\tilde{v}}$ approximately compress to $\tilde{u}, \tilde{v}$. 
This works like in Lemma \ref{lem:DACU}, but we repeat the details to make sure. 
We again choose $N$ such that $N+1 < \delta^{-\frac{1}{2}} < 2N+1$ and we put $\xi_N:=\frac{1}{\sqrt{2N+1}}\sum_{k=-N}^N e_k \in \ell^2(\bZ)$ as in Lemma \ref{lem:DACU}. 
We compress $\tilde{\tilde{u}}, \tilde{\tilde{v}}$ with respect to $\iota:=\mathrm{id}\otimes\xi_N\colon \cH\otimes \ell^2(\mathbb Z) \to \cH\otimes \ell^2(\mathbb Z)\otimes \ell^2(\mathbb Z)$
and obtain as before
\[
\iota^* \tilde{\tilde{u}}\iota =\tilde{u}\langle S \xi_N,\xi_N\rangle = \frac{2N}{2N+1} \tilde{u}
\]
and
\[
\iota^*\tilde{\tilde{v}}\iota =\frac{1}{2N+1}\sum_{m=-N}^N  q^{-m} (u \otimes \mathrm{id})^{-m} \tilde{v} (u \otimes \mathrm{id})^{m} .
\]
This gives
\[
\|\tilde{u}-\iota^*\tilde{\tilde{u}} \iota\| = \frac{1}{2N+1}<\delta^{\frac{1}{2}}.
\]
Now observe that
\begin{align*}
\|\tilde{v} - \iota^*\tilde{\tilde{v}}\iota\| 
&\leq \sup_k \left\|q^k u^k v u^{-k} - \frac{1}{2N+1} \sum_{m=-N}^N q^{k-m} u^{k-m} v u^{m-k} \right\| \\
&= \left\|v - \frac{1}{2N+1} \sum_{m=-N}^N q^{-m} u^{-m} v u^{m} \right\| 
\end{align*}
but the last expression was already estimated in the course of the proof of Lemma \ref{lem:DACU}, and we conclude that 
\[
\|\tilde{v} - \iota^*\tilde{\tilde{v}}\iota\| < \delta^{1/2}.
\]
\end{proof}

The ``reverse" construction of Lemma \ref{lem:DACU_inverse} can also be iterated. 
If $\lambda \in \bC^d$ and $U \in \cU(d)$ then we write $\lambda U$ as a shorthand for $(\lambda_1 U_1, \ldots, \lambda_d U_d)$. 

\begin{lemma}\label{lem:Theta-commuting_dilation_inverse}
Let $Q$ and $U$ as in Lemma \ref{lem:Theta-commuting_dilation}, and let $\hat{U}$ be the unitary tuple constructed in the proof of that lemma.
Define the unitary tuple $\hat{\hat{U}}$ on $\cH \otimes \ell^2(\bZ) \otimes \ell^2(\bZ)$ by
\[
\hat{\hat{U}}_m = \hat{U}_m \otimes S = U_m \otimes S \otimes S
\]
and, for $j \neq m$, 
\begin{align*}
\hat{\hat{U}}_j 
&= \sum_r q_{j,m}^{-r} (U_m \otimes \mathrm{id})^{-r} \hat{U}_j (U_m \otimes \mathrm{id})^{r} \otimes p_r \\
&= \sum_{r,k}  q_{j,m}^{k-r} U_m^{k-r}U_jU_m^{r-k}\otimes p_{k} \otimes p_r  .
\end{align*}
Then the tuple $\hat{\hat{U}}$ is equivalent to a direct sum $\oplus_{\lambda\in \Lambda} (\lambda U)$ where $\Lambda$ is a sequence of points in $\bT^d$. 
Moreover
\[
\distance_{\mathrm{rD}}(\hat{U} \to \hat{\hat{U}}) < \sqrt{\delta}. 
\]
\end{lemma}
\begin{proof}
For every $\ell \in \bN$ let 
\[
\cH_\ell = \ol{\spn}\{\cH \otimes e_k \otimes e_{k + \ell} \colon k \in \bZ\} \subseteq \cH \otimes \ell^2(\bZ) \otimes \ell^2(\bZ).
\]
Then clearly every $\cH_\ell$ is a reducing subspace for $\hat{\hat{U}}$ and 
\[
\cH \otimes \ell^2(\bZ) \otimes \ell^2(\bZ) = \oplus_{\ell \in \bZ} \cH_\ell . 
\]
We can think of $\cH_\ell$ as $\cH \otimes \ell^2(\bZ)$, and as such it carries a natural shift 
\[
\mathrm{id}_\cH \otimes S_\ell \colon h \otimes e_k \otimes e_{k+\ell} \mapsto h \otimes e_{k +1}\otimes e_{k+1+\ell} .
\]
On $\cH_\ell$ the tuple $\hat{\hat{U}}$ reduces to the tuple comprised of $U_m \otimes S_\ell$ and the unitaries $q_{j,m}^{-\ell} U_m^{-\ell} U_j U_m^\ell \otimes \mathrm{id}_{\ell^2(\bZ)}$ for $j\neq m$ (in the natural order). 
We see that the tuple $\hat{\hat{U}}$ is unitarily equivalent to the direct sum
\[
\bigoplus_\ell (U_1\otimes q_{1,m}^{-\ell}{\id}, \ldots, U_{m-1}\otimes q_{m-1,m}^{-\ell}{\id},  U_m \otimes S_\ell, U_{m+1}\otimes q_{m+1,m}^{-\ell}{\id} ,\ldots, , U_{d}\otimes q_{d,m}^{-\ell}{\id})
\]
which readily implies that it is equivalent to $\oplus_{\lambda  \in \Lambda} \lambda U$ where $\Lambda$ is a sequence of points in $\bT^d$.

To show that $\hat{\hat{U}}$ approximately compresses to $\hat{U}$ works like in Lemmas  \ref{lem:DACU}, \ref{lem:Theta-commuting_dilation} and \ref{lem:DACU_inverse}. 
\end{proof}

\begin{theorem}\label{thm:Utag}
Let $Q = (\exp(i\theta_{k,\ell}))$ where $\Theta=(\theta_{k,\ell})$ is a real $d\times d$ antisymmetric matrix and let $\delta \in (0,1)$. 
 Assume that $U \in \cU(d)$ such that $\|U_iU_{j}-q_{j,i}U_i U_j\| \leq \delta$ for $1\leq i,j< {d}$. 
Then there  exist $V_\Theta,U'\in \cU(d)$ such that
\begin{itemize}
\item  $V_\Theta$ is $\Theta$-commuting
\item  $U'\sim\bigoplus_{\lambda\in\Lambda} \lambda U$ for a countable set $\Lambda\subset\mathbb T^d$
\item $\distance_{\mathrm{rD}}(U\to V_\Theta) < (d-1)\sqrt{\delta}$
\item  $\distance_{\mathrm{rD}}(V_\Theta\to U') < (d-1)\sqrt{\delta}$.
\end{itemize}
\end{theorem}

\begin{proof}
  Define $U^{(k)}$ as the $k-1$-fold iteration of Lemma \ref{lem:Theta-commuting_dilation}, so that the first $k$ operators in $U^{(k)}$ exactly commute according to $Q$ while the norm of the other commutators remains bounded by $\delta$. Put $V_\Theta:=U^{(d)}$. Then $V_\Theta$ is $\Theta$-commuting and, by triangle inequality, $\distance_{\mathrm{rD}}(U\to V_\Theta)< (d-1)\sqrt{\delta}$.
  Now we apply Lemma \ref{lem:Theta-commuting_dilation_inverse} to $U^{(k+1)}=\hat U^{(k)}$, the $k$th step of the iteration, and obtain
  \[\hat{\hat U}^{(k)}\sim \bigoplus_{\lambda_{k}\in \Lambda_{k}}\lambda_{k} U^{(k)},\quad \distance_{\mathrm{rD}}(U^{(k+1)}\to \hat{\hat{U}}^{(k)})<\sqrt{\delta}\]
  for some countable subset $\Lambda_{k}\subset \mathbb T^d$. Put $V^{(d)}:=U^{(d)}$ and, for $k<d$,
  \[V^{(k)}:=\bigoplus_{\substack{\ell\in\{k,\ldots,d-1\}\\\lambda_{\ell}\in \Lambda_{\ell}}}\lambda_{k}\cdots\lambda_{d-1} U^{(k)}\sim \bigoplus_{\substack{\ell\in\{k+1,\ldots,d-1\}\\\lambda_{\ell}\in \Lambda_{\ell}}}\lambda_{k+1}\cdots\lambda_{d-1} \hat{\hat{U}}^{(k)}.\]
  Note that, for arbitrary countable families of operator tuples $A_i\in B(\mathcal H_i^d),B_i\in B(\mathcal K_i^d)$, we have
  \begin{multline*}\textstyle
    \distance_{\mathrm{rD}}(\bigoplus_i A_i \to \bigoplus_i B_i)=\inf_{\Psi(B)\prec B}\left\|\bigoplus_i A_i-\Psi\left(\bigoplus_i B_i\right)\right\|\\\leq\sup_i\inf_{\Psi_i(B_i)\prec B_i}\|A_i-\Psi_i(B_i)\|=\sup_i \distance_{\mathrm{rD}}(A_i \to B_i)
  \end{multline*}
  Also, $\distance_{\mathrm{rD}}(\lambda A \to \lambda B)=\distance_{\mathrm{rD}}(A\to B)$ for all $\lambda\in\mathbb T^d,A\in B(\mathcal H^d),B\in\mathcal B(\mathcal K^d)$, obviously.
  Therefore,
  \[\distance_{\mathrm{rD}}(V^{(k+1)}\to V^{(k)})\leq \distance_{\mathrm{rD}}(U^{(k+1)}\to \hat{\hat{U}}^{(k)})<\sqrt{\delta}\]
  and, by the triangle inequality, $U':=V^{(1)}$ has all properties claimed in the theorem.
\end{proof}

\begin{corollary}\label{cor:gi_dist}
If $U$ is as in Theorem \ref{thm:Utag} and is gauge invariant, then 
\[
\distance_{\mathrm{rD}}(U,U_\Theta) < (d-1) \sqrt{\delta} \,\, \textrm{ and } \,\, \distance_{\mathrm{HR}}(U,U_\Theta) < 10 \sqrt{(d-1) \sqrt{\delta}}.
\] 
\end{corollary}
\begin{proof}
If $U$ is gauge invariant then so is $V_\Theta$ and constructed in the proof of Theorem \ref{thm:Utag}, hence $V_\Theta \sim U_\Theta$. 
Moreover if $U$ is gauge invariant then $U'\sim\bigoplus_{\lambda\in\Lambda} \lambda U \sim U$. 
It follows from the theorem that $\distance_{\mathrm{rD}}(U,U_\Theta) < (d-1) \sqrt{\delta}$ and therefore $\distance_{\mathrm{HR}}(U,U_\Theta) < 10 \sqrt{(d-1) \sqrt{\delta}}$ by Theorem \ref{thm:dHRleqdD}. 
\end{proof}

\subsection{Almost gauge invariant tuples}

We say that a unitary tuple $U$ is \emph{$\varepsilon$-almost gauge invariant} if $\mathrm{d_{mr}}(U,U\otimes z) \leq \varepsilon$. 
In Corollary \ref{cor:gi_dist} above we saw that gauge invariant almost $\Theta$-commuting tuples can be approximated by a $\Theta$-commuting tuple. 

\begin{lemma}\label{lem:almost_inv}
A tuple $U$ is $\varepsilon$-almost gauge invariant if and only if $\mathrm{d_{mr}}(U,\lambda U) \leq \varepsilon$ for every $\lambda \in \bT^d$. 
\end{lemma}
\begin{proof}
This follows from the definitions.
\end{proof}

We will see that an almost $\Theta$-commuting unitary tuple is $\varepsilon$-almost gauge invariant, where $\varepsilon$ depends on $\Theta$. 
We shall use this to show that an almost $\Theta$-commuting unitary tuple can be approximated to a certain extent by a $\Theta$-commuting unitary tuple. 

As usual, let $\Theta=(\theta_{k,\ell})$ be a real antisymmetric $d \times d$ matrix and let $Q = (q_{k,\ell})_{k,\ell=1}^d$ where $q_{k,\ell} = \exp(i\theta_{k,\ell})$. 
We let $\bS_Q$ (or simply $\bS$ when no confusion can arise) be the subgroup of $\bT^d$ generated by the columns $\{q_{*,\ell} : \ell=1,\ldots, d\}$ of $Q$. 
We also define 
\[
\eta_Q = \distance_{\mathrm{H}}(\bT^d,\bS_Q), 
\]
where the Hausdorff distance between two subsets of $\bC^d$ is computed with respect to the underlying $\ell^\infty$ norm on $\bC^d$ given by $\|\lambda - \mu\|_\infty = \max_i|\lambda_i - \mu_i|$. 
We will say that $Q$ (or $\Theta$) is {\em ergodic} if $\eta_Q = 0$, that is, if $\bS_Q$ is dense in $\bT^d$. 

To illustrate, consider the case $d = 2$. In this case $Q = \left(\begin{smallmatrix} 1 & q \\ q^{-1} & 1 \end{smallmatrix} \right)$, and it is easy to see that $\bS_Q$ is finite if and only if $\eta_Q > 0$ and this happens if and only if $q$ is a root of unity. 
In other words, $Q$ is ergodic if and only if $q$ is not a root of unity. 

For $N \in \bN$ we let $\bS_Q(N)$ denote the subset of $\bS_Q$ that is generated by words of length at most $N$ in the columns $\{q_{*,\ell} : \ell=1,\ldots, d\}$ of $Q$ and their inverses. 
Given $\eta > \eta_Q$, we let $N_\eta = N_{Q,\eta}$ be the least integer such that 
\be\label{eq:Ndelta}
\distance_{\mathrm{H}}(\bT^d,\bS_Q(N_\eta)) <\eta. 
\ee

\begin{proposition}\label{prop:Neta}
Let $\Theta$ and $A$ be as above, and let $\eta > \eta_Q$. 
Suppose that $U \in \cU(d)$ is $\delta$-almost $\Theta$-commuting for some $\delta > 0$. 
Then $U$ is $\varepsilon$-almost gauge invariant for 
\[
\varepsilon = N_\eta \delta + \eta. 
\]
\end{proposition}
\begin{proof}
For every $\ell = 1, \ldots, d$, we have an inner $*$-automorphism $\alpha_\ell \colon C^*(U) \to C^*(U)$ given by 
\[
\alpha_\ell(A) = U_\ell A U_\ell^*. 
\]
Applying this to $U_k$ we find that 
\[
\|\alpha_\ell(U_k) - e^{i\theta_{k,\ell}}U_k \| \leq \delta, 
\]
where we have made use of \eqref{eq:del_almost_T_comm}. 
If we write $g_1, \ldots, g_d$ for the columns of $Q$, it follows by induction that for every $N$, every $M \leq N$ and every $\lambda = g_{i_1} \cdots g_{i_M} \in \bS_Q(N)$, 
\[
\|\alpha_{i_1} \circ \cdots \alpha_{i_M} (U) - \lambda U \| \leq N \delta, 
\] 
by which we mean
\[
\|\alpha_{i_1} \circ \cdots \alpha_{i_M} (U_k) - \lambda_k U_k\| \leq N \delta 
\] 
for all $k=1, \ldots, d$. 
By applying a UCP map $\phi \colon C^*(U) \to M_n$ to the above inequality, it follows that for every $X \in \cW(U)$ and every $\lambda \in \bS_Q(N)$, 
\[
\distance(\lambda X, \cW(U)) := \inf \{\|\lambda X - Y \| : Y \in \cW(U) \} \leq N \delta .
\]
Now let $\mu \in \bT^d$ and $X \in \cW(U)$. 
If $\lambda \in \bS_Q(N_\eta)$ is such that $\|\lambda - \mu\|_\infty < \eta$, then 
\[
\distance(\mu X, \cW(U)) \leq \|\mu X - \lambda X \| + \distance(\lambda X, \cW(U)) < \eta + N_\eta \delta. 
\]
From this combined with Lemma \ref{lem:almost_inv} we find that $U$ is $\varepsilon$-almost gauge invariant for $\varepsilon = \eta + N_\eta \delta$, as required. 
\end{proof}


\subsection{The main result}

\begin{theorem}\label{thm:main1}
Let $\Theta=(\theta_{k,\ell})$ be a real $d\times d$ antisymmetric matrix and let $\delta, \varepsilon \in (0,1)$. 
Assume that $U \in \cU(d)$ is $\delta$-almost $\Theta$-commuting and is $\varepsilon$-almost gauge invariant. 
Then there  exists $V_\Theta \in \cU(d)$ such that is $\Theta$-commuting such that $\distance_{\mathrm{rD}}(U,V_\Theta) < \ep + (d-1)\sqrt{\delta}$, and, consequently, 
\be\label{eq:HRleq}
\distance_{\mathrm{HR}}(U,V_\Theta) < 10 \sqrt{\ep + (d-1)\sqrt{\delta}} . 
\ee
\end{theorem}
\begin{proof}
Apply Theorem \ref{thm:Utag} to $U$ in order to find  $V_\Theta,U'\in \cU(d)$ such that
\begin{itemize}
\item  $V_\Theta$ is $\Theta$-commuting
\item  $U'\sim\bigoplus_{\lambda\in\Lambda} \lambda U$ for a countable set $\Lambda\subset\mathbb T^d$
\item $\distance_{\mathrm{rD}}(U\to V_\Theta) < (d-1)\sqrt{\delta}$
\item  $\distance_{\mathrm{rD}}(V_\Theta\to U') < (d-1)\sqrt{\delta}$.
\end{itemize}
By $\ep$-almost gauge invariance, $\mathrm{d_{mr}}(U,\lambda U) \leq \varepsilon$ for every $\lambda \in \bT^d$, therefore $\mathrm{d_{mr}}(U,U') \leq \varepsilon$. 
It follows that $\distance_{\mathrm{rD}}(V_\Theta\to U) < \ep + (d-1)\sqrt{\delta}$, and we conclude that 
\[
\distance_{\mathrm{rD}}(U, V_\Theta) < \ep + (d-1)\sqrt{\delta}. 
\]
Finally, \eqref{eq:HRleq} now follows from Theorem \ref{thm:dHRleqdD}. 
\end{proof}

One can combine Theorem \ref{thm:main1} and Proposition \ref{prop:Neta} with Remark \ref{rem:inf_amp} to obtain the following corollary  
\begin{corollary}\label{cor:main1}
Let $U$ be a $\delta$-almost $\Theta$-commuting unitary $d$-tuple on a Hilbert space $\cH$. 
For all $\eta > \eta_Q$, there exists a $\Theta$-commuting unitary tuple on $\cH \otimes \ell^2$ such that 
\be\label{eq:endgame}
\|U \otimes \mathrm{id}_{\ell^2} - V_\Theta\| < 10 \sqrt{N_\eta \delta + \eta + (d-1)\sqrt{\delta}} . 
\ee
\end{corollary} 

Finally, our main result. 
\begin{theorem}\label{thm:main2}
Let $\Theta$ be an ergodic real antisymmetric $d \times d$ matrix. 
For every $\ep > 0$, there exists $\delta > 0$, such that for every $\delta$-almost commuting unitary tuple $U$ on $\cH$ there exists a $\Theta$-commuting unitary tuple on $\cH \otimes \ell^2$ such that 
\[
\|U \otimes \mathrm{id}_{\ell^2} - V_\Theta\| < \ep .
\]
\end{theorem} 
\begin{proof}
If $\Theta$ is ergodic then $\eta_Q = 0$. 
Let $\eta < \ep^2/200$. This determines the integer $N_\eta$.
Choosing $\delta > 0$ such that $N_\eta \delta + (d-1)\sqrt{\delta} < \ep^2/200$, we take another look at \eqref{eq:endgame} and notice that the proof is complete. 
\end{proof}
Obviously, the above corollary in the special case that $d = 2$ is precisely the final assertion made in the abstract of the paper. 
For a non ergodic matrix $\Theta$, $\eta_Q > 0$ and \eqref{eq:endgame} becomes 
\[
\|U \otimes \mathrm{id}_{\ell^2} - V_\Theta\| < 10 \sqrt{N_{\eta_Q} \delta + \eta_Q + (d-1)\sqrt{\delta}} . 
\]

\begin{example}\label{ex:standard_q}
Let $q = e^{i \theta} \in \bT$ with $\frac{\theta}{2 \pi} \notin \bQ$. 
Fix $\varepsilon$ and the corresponding $\delta$ from Theorem \ref{thm:main2}. 
If $r \in \bT$ satisfies $|q-r| < \delta$, then every $r$-commuting pair $u,v$ is $\delta$-almost $q$-commuting, therefore there is a $q$-commuting pair $u_q, v_q$ such that 
\[
\|u \otimes \mathrm{id}_{\ell^2} - u_q\| + \|v \otimes \mathrm{id}_{\ell^2} - v_q\| < 2 \varepsilon. 
\]
It is natural to ask whether the ampliation is necessary. 
As remarked in the introduction, if $r$ is a root of unity and $u,v$ are represented on a finite dimensional space $\bC^n$, then there is no $q$-commuting pair $u_q,v_q$ on $\bC^n$, and in particular no $q$-commuting pair can approximate the unampliated pair $u,v$. 

On the other hand, in some cases, an ampliation is not necessary. 
Consider the pair $S,D_r$, where $S$ be the bilateral shift on $\ell^2(\bZ)$, and $D_r$ is the diagonal operator determined on the standard basis of $\ell^2(\bZ)$ by $D_r e_n = r^n e_n$. 
The C*-algebra generated by $S$ and $D_r$ contains no nonzero compacts, therefore Voiculescu's theorem (see \cite[Corollary II.5.6]{DavBook}) implies that the pair $S,D_r$ is approximately unitarily equivalent to its infinite ampliation. 
Invoking Theorem \ref{thm:main2}, we find that if $|q-r|<\delta$ then there is a pair of $q$-commuting unitaries $u_q,v_q \in B(\ell^2(\bZ))$ such that 
\[
\|S - u_q\| + \|D_r - v_q\| < 2 \varepsilon. 
\]
\end{example}

\ 

\

\noindent {\bf Acknowledgements.} We are grateful for enlightening correspondences with Ken Davidson, Adam Dor-On, Terry Loring and Huaxin Lin. 

\bibliographystyle{myplainurl}
\bibliography{GS}

\linespread{1}
\setlength{\parindent}{0pt}

\end{document}